\documentclass{article}
\usepackage{amsmath,amssymb,amsfonts}
\usepackage{amsthm}
\usepackage{geometry}
\usepackage[hidelinks]{hyperref}

\newtheorem*{abstract*}{Abstract}
\newtheorem{theorem}{Theorem}[section]
\newtheorem{corollary}[theorem]{Corollary}
\newtheorem{lemma}[theorem]{Lemma}
\newtheorem{remark}[theorem]{Remark}

\newcommand{\be}{\begin{equation}}
\newcommand{\ee}{\end{equation}}
\newcommand{\bd}{\begin{displaymath}}
\newcommand{\ed}{\end{displaymath}}
\everymath{\displaystyle}
\input{tcilatex}
\begin{document}

\title{{\huge \textbf{Non-spurious solutions to second order BVP by
monotonicity methods}}}
\author{Filip Pietrusiak}
\date{}
\maketitle
\begin{abstract*}
\normalfont We consider the following BVP $\ddot{x}\left( t\right) =f\left(
t,\dot{x}\left( t\right) ,x\left( t\right) \right) -h\left( t\right) $, $%
x\left( 0\right) =x\left( 1\right) =0$, where $f$ is continuous and
satisfies some other conditions, $h\in H_{0}^{1}\left( 0,1\right) $ together
with its discretization 
\begin{equation*}
-\Delta ^{2}x(k-1)+\frac{1}{n^{2}}f\left( \frac{k}{n},n\Delta x\left(
k-1\right) ,x\left( k\right) \right) =\frac{1}{n^{2}}h\left( \frac{k}{n}%
\right) \text{ for }k\in \left\{ 1, 2, \ldots ,n\right\} \text{.}
\end{equation*}%
Using monotonicity methods we obtain the convergence of a solutions to a
family of discrete problems to the solution of a continuous one, i.e. the
existence of non-spurious solutions to the above problems is considered.
Continuous dependence on parameters for the continuous problem is also
investigated.\newline
\textbf{Keywords:} non-spurious solutions, monotonicity methods, continuous
dependence on parameters, boundary value problems.\newline
\textbf{Mathematics Subject Classification:} 39A12, 39A10, 30E25, 34B15.
\end{abstract*}
\section{Introduction}
In this note we consider non-spurious solutions by using a monotonicity
methods to the following second order BVP 
\begin{equation}
\left\{ 
\begin{array}{l}
\ddot{x}\left( t\right) =f\left( t,\dot{x}\left( t\right) ,x\left( t\right)
\right) -h\left( t\right) , \\ 
x\left( 0\right) =x\left( 1\right) =0,%
\end{array}%
\right.  \label{problem}
\end{equation}%
where $f:\left[ 0,1\right] \times \mathbb{R}^{2}\rightarrow \mathbb{R}$ is a
continuous function such that $f\left( t,0,0\right) =0$ and $h:\left[ 0,1%
\right] \rightarrow \mathbb{R}$ is a continuous function such that $h\left(
0\right) =h\left( 1\right) =0$. We will make precise the background further
on. The following assumptions will be used in this work 
\begin{eqnarray*}
&\text{\textit{\textbf{P1}}}&\forall _{r>0}\exists _{f_{r}\in L^{1}\left(
0,1\right) }\forall _{x\in H_{0}^{1}\left( 0,1\right) }\left\Vert
x\right\Vert \leq r\Rightarrow \left\vert f\left( t,\dot{x}\left( t\right)
,x\left( t\right) \right) \right\vert \leq f_{r}\left( t\right) \text{ a.e.
in }\left( 0,1\right)\text{,} \\
&\text{\textit{\textbf{P2}}}&\forall _{s,t,w,z\in \mathbb{R}\text{, }k,l\in
\lbrack 0,1]}\left( s-t\right) \left( f\left( k,w,s\right) -f\left(
l,z,t\right) \right) \geq 0\text{.}
\end{eqnarray*}%
Condition \textit{\textbf{P1}} is assumed in order to make sure that
suitable operator, which we will use, is well defined, while \textit{\textbf{%
P2}} is assumed in order to apply monotonicity methods.

Together with problem \eqref{problem} we consider its discretization defined
as follows. For fixed $n\in \mathbb{N}$ we consider the following
discretization from \cite{gaines} p. 411 
\begin{equation}
-\Delta ^{2}x(k-1)+\frac{1}{n^{2}}f\left( \frac{k}{n},n\Delta x\left(
k-1\right) ,x\left( k\right) \right) =\frac{1}{n^{2}}h\left( \frac{k}{n}%
\right) \text{ for }k\in \left\{ 1, 2, \ldots,n\right\} \text{,}  \label{dys}
\end{equation}%
where $x:[0,n]\cap \mathbb{N}_{0}\rightarrow \mathbb{R}$, $f$ and $h$ have
the same properties as above and $x\left( 0\right) =x\left( n\right) =0$.
Again solutions are understood in the weak sense which will be made precise
further; $\mathbb{N}_{0}:=\mathbb{N}\cup \left\{ 0\right\} .$ \newpage
Assume that both continuous boundary value problem \eqref{problem} and for
each fixed $n\in \mathbb{N}$ the discrete boundary value problem \eqref{dys}
are solvable by $x$ and $x^{n}=\left( x^{n}(k)\right) $, respectively.
Moreover, let there exist two positive constants $Q$, $N$ such that 
\begin{equation}
n|\Delta x^{n}(k-1)|\leq Q\text{ and }|x^{n}(k)|\leq N  \label{estimation}
\end{equation}%
for all $k=1,2,\ldots ,n$ and all $n\geq n_{0}$, where $n_{0}$ is fixed (and
possibly arbitrarily large). Lemma 2.4 from \cite{gaines} p. 414 says that
for some subsequence $\left( x^{n_{m}}\right) _{m\in \mathbb{N}}$ of $\left(
x^{n}\right) _{n\in \mathbb{N}}$ it holds 
\begin{equation}
\lim_{m\rightarrow \infty }\max_{0\leq k\leq n_{m}}\left\vert \bar{x}%
^{n_{m}}\left( k\right) -x\left( \frac{k}{n_{m}}\right) \right\vert =0\text{%
, }\lim_{m\rightarrow \infty }\max_{0\leq k\leq n_{m}}\left\vert \bar{v}%
^{n_{m}}\left( k\right) -\dot{x}\left( \frac{k}{n_{m}}\right) \right\vert =0%
\text{,}  \label{main}
\end{equation}%
where 
\begin{equation*}
\bar{x}^{n}\left( t\right) :=x^{n}\left( k\right) +n\Delta x^{n}\left(
k\right) \left( t-\frac{k}{n}\right) \text{, for }\frac{k}{n}\leq t<\frac{k+1%
}{n}\text{,}
\end{equation*}%
\begin{equation*}
\bar{v}^{n}\left( t\right) :=\left\{ 
\begin{array}{ll}
n\Delta x^{n}\left( k-1\right) +n^{2}\Delta ^{2}x^{n}\left( k-1\right)
\left( t-\frac{k}{n}\right) \text{, } & \frac{k}{n}\leq t<\frac{k+1}{n}\text{%
,} \\ 
n\Delta x^{n}\left( 0\right) \text{, } & 0\leq t<\frac{1}{n}\text{.}%
\end{array}%
\right.
\end{equation*}%
Note that if both continuous and discrete problem have unique solution then
the convergence holds for the whole sequence, see the last comments in paper 
\cite{Tisdell}. Now following comments in \cite{non-spurious}, we introduce
the idea of a non-spurious solution. The solutions of a family of problems %
\eqref{dys} which converge to some solution of problem \eqref{problem} in
the sense described by relation \eqref{main} are addressed as non-spurious
solutions.

There have been some research in the area of non-spurious solutions
addressing mainly problems whose solutions where obtained by the fixed point
theorems and the method of lower and upper solutions, \cite{rech1}, \cite%
{rachunkowa2}, \cite{Tisdell}. In \cite{non-spurious} the variational method
is applied, namely the direct method of the calculus of variations. In this
note we are aiming at using monotonicity method in order to show that in
this setting one can also obtain suitable convergence results. While the
approach is somewhat similar to this of \cite{non-spurious}, we see that $f$
in contrast to \cite{non-spurious} can be dependent on the derivative and as
an additional result we get continuous dependence on parameters which seems
to be of some novelty by monotonicity approach. As expected we will have to
get the uniqueness of solutions for the associated discrete problem, which
is not always easy to be obtained, see \cite{uni2}. Moreover, in our case
the first estimation in \eqref{estimation} does not follow from the second
one, as the case with $f$ not depending on $\Delta x$, so that it must be
derived from the conditions imposed on our problem. As appears with other
methods, monotonicity and boundedness of solutions are inherited in the
discrete problem from the continuous one. To prove the existence and
uniqueness of solution in \eqref{problem} and for fixed $n\in \mathbb{N}$ in %
\eqref{dys}, we need following Corollary 6.1.9 from \cite{Drabek} p. 370.
\begin{corollary}
\label{drab} Let $H$ be a real Hilbert space and $T:H\rightarrow H$ be a
continuous and strongly monotone operator, i.e. there exists some $c>0$ that%
\begin{equation}
\left\langle Tx-Ty,x-y\right\rangle \geq c \left\Vert x-y\right\Vert ^{2}
\label{definition-strongly_monotone}
\end{equation}%
for all $x,y\in H$. Then for any $h\in H$ the equation 
\begin{equation*}
T\left( u\right) =h
\end{equation*}%
has a unique solution. Let $T\left( u_{1}\right) =h_{1}$ and $T\left(
u_{2}\right) =h_{2}$. Then 
\begin{equation*}
\left\Vert u_{1}-u_{2}\right\Vert \leq \frac{1}{c}\left\Vert
h_{1}-h_{2}\right\Vert
\end{equation*}%
with $c>0$ defined in \eqref{definition-strongly_monotone}, i.e. $T^{-1}$ is
Lipschitz continuous.
\end{corollary}
\section{Non-spurious solutions for (\protect\ref{problem})}
\subsection{The continuous problem}
Solutions of \eqref{problem} will be investigated in the real Hilbert space $%
H_{0}^{1}\left( 0,1\right) $ consisting of absolutely continuous functions
satisfying the boundary conditions, which have an a.e. derivative being
integrable with square. In the space $H_{0}^{1}\left( 0,1\right)$ we
introduce following norm 
\begin{equation*}
\left\Vert x\right\Vert :=\left( \int_{0}^{1}{\left( \dot{x}\left( t\right)
\right) ^{2}dt}\right) ^{\frac{1}{2}}
\end{equation*}%
and with a natural scalar product given by 
\begin{equation*}
\left\langle x,y\right\rangle := \int_{0}^{1}{\dot{x}\left( t\right) \dot{y}%
\left( t\right) dt}\text{.}
\end{equation*}%
Symbol $\left\Vert \cdot \right\Vert $ will always denote the norm in $%
H_{0}^{1}\left( 0,1\right) $, while for other norms we shall write
explicitly. Since we apply monotonicity methods, we look only for $%
H_{0}^{1}\left( 0,1\right) $ solutions which are called the weak solutions.
A function $x\in H_{0}^{1}\left( 0,1\right) $ is a weak $H_{0}^{1}\left(
0,1\right) $ solution to \eqref{problem}, if the following equality 
\begin{equation}
\int_{0}^{1}{\dot{x}\left( t\right) \dot{y}\left( t\right) dt}+\int_{0}^{1}{%
y\left( t\right) f\left( t,\dot{x}\left( t\right) ,x\left( t\right) \right)
dt}=\int_{0}^{1}{y\left( t\right) h\left( t\right) dt}  \label{weak}
\end{equation}%
holds for all $y\in H_{0}^{1}\left( 0,1\right) $, see \cite{brezis} p. 201.
In order to obtain \eqref{weak} one multiplies the given equation %
\eqref{problem} by a test function from $H_{0}^{1}\left( 0,1\right) $ and
takes integrals. Next we use integration by parts.

Now we must prove that integrals which arise in our problem are finite for
any fixed $x,y\in H_{0}^{1}\left( 0,1\right) $. Firstly $\int_{0}^{1}{\dot{x}%
\left( t\right) \dot{y}\left( t\right) dt}$ is finite since $x$, $y\in
H_{0}^{1}\left( 0,1\right) $. Secondly $\int_{0}^{1}{y\left( t\right)
h\left( t\right) dt}$ is finite, because of continuity of both $h$ and $y$.
The most demanding is $\int_{0}^{1}{y\left( t\right) f\left( t,\dot{x}\left(
t\right) ,x\left( t\right) \right) dt}$. By \textit{\textbf{P1}} we see that 
\begin{equation*}
\left\vert \int_{0}^{1}{y\left( t\right) f\left( t,\dot{x}\left( t\right)
,x\left( t\right) \right) dt}\right\vert \leq \int_{0}^{1}{\left\vert
y\left( t\right) \right\vert \left\vert f\left( t,\dot{x}\left( t\right)
,x\left( t\right) \right) \right\vert dt}\leq \int_{0}^{1}{\left\vert
y\left( t\right) \right\vert f_{r}\left( t\right) dt}\leq c,
\end{equation*}%
where $c>0$ is some constant. We note that any solution of our problem is in
fact classical one. Indeed, let us recall the following well know regularity
tool, i.e. the du Bois-Reymond Lemma from \cite{MawhinProblemmes}.
\begin{lemma}
\label{fundamentalLemma}If $g\in L^{1}\left( 0,1\right) ,$ $h\in L^{2}\left(
0,1\right) $ and 
\begin{equation}
\int_{0}^{1 }\left( g\left( t\right) y\left( t\right) +h\left( t\right) \dot{%
y}\left( t\right) \right) dt=0  \label{Lem_Fund_rel}
\end{equation}%
for all $y\in H_{0}^{1}\left( 0,1\right) $, then $\dot{h}=g$ a.e. on $\left[
0,1\right] $ and $\dot{h}\in L^{1}\left( 0,1\right) $.
\end{lemma}
We note that a function $g$ satisfying \eqref{Lem_Fund_rel} is a weak
derivative of a function $h$, see \cite{brezis} p. 202, for the definition.
From Lemma \ref{fundamentalLemma}\ it follows that $h\left( t\right)
=\int_{0}^{t}g\left( s\right) ds+c$ for some constant $c$ and for a.e. $t\in %
\left[ 0,1\right] $. By standard arguments (see \cite{KufnerFucik}), we see
that $g$ is in fact a classical almost everywhere derivative of $h$. This
observation leads to the conclusion that function any weak solution is such a function from $H^{1}_{0}\left( 0, 1\right)$ that the second derivative exists and it is an $L^{1}\left( 0, 1\right)$ function. Such solutions are called classical solution of problem \eqref{problem}.

To prove the existence and the uniqueness of solution we use monotonicity
methods. This means that we must find monotone operator which is associated
to \eqref{problem} and use Corollary \ref{drab}. We introduce operator $%
K:H_{0}^{1}\left( 0,1\right) \rightarrow H_{0}^{1}\left( 0,1\right) $ such
that 
\begin{equation*}
\left\langle Kx,y\right\rangle :=\int_{0}^{1}{\dot{x}\left( t\right) \dot{y}%
\left( t\right) dt}+\int_{0}^{1}{y\left( t\right) f\left( t,\dot{x}\left(
t\right) ,x\left( t\right) \right) dt}\text{, where }x,y\in H_{0}^{1}\left(
0,1\right) .
\end{equation*}%
We shall see that \textit{\textbf{P2}} implies strong monotonicity of
operator $K$, while continuity of $K$ follows by continuity of $f$. Indeed,
we have following theorem.
\begin{theorem}
\label{existc} Assume that $f:\left[ 0,1\right] \times \mathbb{R}%
^{2}\rightarrow \mathbb{R}$ is continuous, and that conditions \textbf{P1}, 
\textbf{P2} are satisfied. Then problem \eqref{problem} has exactly one
solution.
\end{theorem}
\newpage 
\begin{proof}
Let $f$ satisfies \textit{\textbf{P1}} and \textit{\textbf{P2%
}}. We divide our reasoning into two parts.

\textbf{Monotonicity part}

For all $x,y\in H_{0}^{1}\left( 0,1\right) $ we have%
\begin{equation*}
\begin{split}
\left\langle Kx-Ky,x-y\right\rangle &= \left\langle Kx,x\right\rangle
-\left\langle Kx,y\right\rangle -\left\langle Ky,x\right\rangle
+\left\langle Ky,y\right\rangle \\
&=\int_{0}^{1}{\left( \dot{x}\left( t\right) -\dot{y}\left( t\right) \right)
^{2}dt}+\int_{0}^{1}{\left( x\left( t\right) -y\left( t\right) \right)
\left( f\left( t,\dot{x}\left( t\right) ,x\left( t\right) \right) -f\left( t,%
\dot{y}\left( t\right) ,y\left( t\right) \right) \right) dt} \\
&=\left\Vert x-y\right\Vert ^{2}+\int_{0}^{1}{\left( x\left( t\right)
-y\left( t\right) \right) \left( f\left( t,\dot{x}\left( t\right) ,x\left(
t\right) \right) -f\left( t,\dot{y}\left( t\right) ,y\left( t\right) \right)
\right) dt}
\end{split}%
\end{equation*}
and from \textit{\textbf{P2}} the second summand is non-negative. Hence 
\begin{equation*}
\left\langle Kx-Ky,x-y\right\rangle \geq \left\Vert x-y\right\Vert ^{2}.
\end{equation*}%
Therefore $K$ is strongly monotone with constant $c=1$.

\textbf{Continuity part}

Let us take a sequence $\left( x_{n}\right) _{n\in \mathbb{N}}\subset $ $%
H_{0}^{1}\left( 0,1\right) $ convergent (strongly) in $H_{0}^{1}\left(
0,1\right) $ to some $x_{0}$. This means that both $\left( x_{n}\right)
_{n\in \mathbb{N}}$ and $\left( \dot{x}_{n}\right) _{n\in \mathbb{N}}$ are
convergent in $L^{2}\left( 0,1\right) $. Of course $\left( x_{n}\right)
_{n\in \mathbb{N}}$ is bounded, i.e. exists constant $r>0$ such that $%
\left\Vert x_{n}\right\Vert <r$ for $n\in \mathbb{N}$. Let us fix $y\in
H_{0}^{1}\left( 0,1\right) $. The term $x\rightarrow \int_{0}^{1}{\dot{x}%
\left( t\right) \dot{y}\left( t\right) dt}$ is obviously continuous. We see
from \textit{\textbf{P1}} that 
\begin{equation*}
\int_{0}^{1}{\left\vert y\left( t\right) \right\vert \left\vert f\left( t,%
\dot{x}_{n}\left( t\right) ,x_{n}\left( t\right) \right) \right\vert dt}\leq
\int_{0}^{1}{\left\vert y\left( t\right) \right\vert f_{r}\left( t\right) dt}%
\text{.}
\end{equation*}%
Then by the Lebesgue Dominated Theorem operator $K$ is continuous.
Therefore, we can use a Corollary \ref{drab} and we get existence and
uniqueness of solution to \eqref{problem}. \end{proof}
\subsection{The discrete problem}
Now we consider discretization of problem \eqref{problem}, i.e. problem %
\eqref{dys} with fixed $n\in \mathbb{N}$. The space in which the solutions
are considered is as follows 
\begin{equation*}
\mathbb{E}:=\left\{ x:[0,n]\cap \mathbb{N}_{0}\rightarrow \mathbb{R}\,|\,%
\text{ }x\left( 0\right) =x\left( n\right) =0\right\} \text{.}
\end{equation*}%
Clearly, $\mathrm{dim}\left( \mathbb{E}\right) =n$ and $\mathbb{E}$ is a
Hilbert space. In the $\mathbb{E}$ we choose a norm given by 
\begin{equation*}
\left\Vert x\right\Vert _{\mathbb{E}}:=\left( \sum_{k=1}^{n}{\left( \Delta
x\left( k-1\right) \right) ^{2}}\right) ^{\frac{1}{2}}\text{,}
\end{equation*}%
with the following scalar product 
\begin{equation*}
\left\langle x,y\right\rangle :=\sum_{k=1}^{n}{\Delta x\left( k-1\right)
\Delta y\left( k-1\right) },\quad x,y\in \mathbb{E}.
\end{equation*}%
Since the space $\mathbb{E}$ is finite dimensional the norm $\left\Vert
\cdot \right\Vert _{\mathbb{E}}$ in $\mathbb{E}$ is equivalent to the usual
norm $\left\Vert \cdot \right\Vert _{0}$ 
\begin{equation*}
\left\Vert x\right\Vert _{0}:=\left( \sum_{k=1}^{n}{\left\vert x\left(
k\right) \right\vert ^{2}}\right) ^{\frac{1}{2}}\text{.}
\end{equation*}%
We have also the following inequality 
\begin{equation}
\left\Vert x\right\Vert _{\mathbb{E}}\leq 2\left\Vert x\right\Vert _{0}\text{%
.}  \label{normy}
\end{equation}%
Function $x\in \mathbb{E}$ is a solution to \eqref{dys} provided that 
\begin{equation*}
\sum_{k=1}^{n}{\Delta x\left( k-1\right) \Delta y\left( k-1\right) }+\frac{1%
}{n^{2}}\sum_{k=1}^{n}{y\left( k\right) f\left( \frac{k}{n},n\Delta x\left(
k-1\right) ,x\left( k\right) \right) }=\frac{1}{n^{2}}\sum_{k=1}^{n}y\left(
k\right) h{\left( k\right) }
\end{equation*}%
for all $y\in \mathbb{E}$. Let $n$ be a natural number, $f:\left[ 0,1\right]
\times \mathbb{R}^{2}\rightarrow \mathbb{R}$ be continuous. We define
operator $K:\mathbb{E}\rightarrow \mathbb{E}$ such that 
\begin{equation*}
\left\langle Kx,y\right\rangle :=\sum_{k=1}^{n}{\Delta x\left( k-1\right)
\Delta y\left( k-1\right) }+\frac{1}{n^{2}}\sum_{k=1}^{n}{y\left( k\right)
f\left( \frac{k}{n},n\Delta x\left( k-1\right) ,x\left( k\right) \right) }%
\text{, where }x,y\in \mathbb{E}.
\end{equation*}%
Reasoning exactly as in the proof of Theorem \ref{existc} (the monotonicity
part) given the continuity of $K$ (which is obvious by the continuity of $f$
and since we are now in the finite dimensional setting), we have the
following theorem.
\begin{theorem}
\label{existd} Assume that $f:\left[ 0,1\right] \times \mathbb{R}%
^{2}\rightarrow \mathbb{R}$ is continuous and that condition \textbf{P2} is
satisfied. Then problem \eqref{dys} has exactly one solution.
\end{theorem}
\subsection{Main result}
In this section we prove the existence of non-spurious solutions to our
problem. In order to do this we must first obtain some inequalities
concerning the solutions to the discrete problem \eqref{dys} which would
lead to estimations \eqref{estimation} and further to conclusion \eqref{main}.
\begin{lemma}
Assume that $f:\left[ 0,1\right] \times \mathbb{R}^{2}\rightarrow \mathbb{R}$
is continuous and that \textbf{P2} is satisfied. By $x^{n}$ we denote
solution of discrete problem \eqref{dys} with fixed $n\in \mathbb{N}$ and
fixed $h$. We have the following inequality 
\begin{equation}
\left\Vert x^{n}\right\Vert _{\mathbb{E}}\leq \frac{2}{n^{\frac{3}{2}}}%
\left\Vert h\right\Vert _{C\left( \left[ 0,1\right] \right) }\text{.}
\label{ogr}
\end{equation}
\end{lemma}
\begin{proof}
Let $n\in\mathbb{N}$ and $h\in \mathbb{E}$ be fixed. Then with Corollary \ref{drab}, in which we take $h_{1}=\frac{1}{n^2}h$, $h_{2}=\theta$ and $c=1$.
We have 
\begin{equation*}
\left\Vert x^{n}-x_{\theta }^{n}\right\Vert_{\mathbb{E}} \leq \frac{1}{n^2}%
\left\Vert h\right\Vert_{\mathbb{E}}\text{.}
\end{equation*}%
Additionally 
\begin{equation*}
\left\Vert x^{n}\right\Vert_{\mathbb{E}} -\left\Vert x_{\theta
}^{n}\right\Vert_{\mathbb{E}} \leq \left\Vert x^{n}-x_{\theta
}^{n}\right\Vert_{\mathbb{E}}\text{.}
\end{equation*}%
Therefore, for each $h\in \mathbb{E}$, we have 
\begin{equation}
\left\Vert x^{n}\right\Vert_{\mathbb{E}} \leq \frac{1}{n^2}\left\Vert
h\right\Vert_{\mathbb{E}} +\left\Vert x_{\theta }^{n}\right\Vert_{\mathbb{E}}%
\text{.}  \label{ograniczenie}
\end{equation}%
Moreover, for $y\in \mathbb{E}$ and $h=\theta $, we get 
\begin{equation*}
\begin{split}
\sum_{k=1}^{n}{\Delta x\left( k-1\right) \Delta y\left( k-1\right) }+\frac{1%
}{n^{2}}\sum_{k=1}^{n}{y\left( k\right) f\left( \frac{k}{n},n\Delta x\left(
k-1\right) ,x\left( k\right) \right) }&=\frac{1}{n^{2}}\sum_{k=1}^{n}{%
y\left( k\right) \theta \left( k\right) }\text{,} \\
\sum_{k=1}^{n}{\Delta x\left( k-1\right) \Delta y\left( k-1\right) }+\frac{1%
}{n^{2}}\sum_{k=1}^{n}{y\left( k\right) f\left( \frac{k}{n},n\Delta x\left(
k-1\right) ,x\left( k\right) \right) }&=0\text{.}
\end{split}%
\end{equation*}
For $x=\theta$. We have%
\begin{equation*}
\begin{split}
\sum_{k=1}^{n}{\Delta \theta \left( k-1\right) \Delta y\left( k-1\right) }+%
\frac{1}{n^{2}}\sum_{k=1}^{n}{y\left( k\right) f\left( \frac{k}{n},n\Delta
\theta \left( k-1\right) ,\theta \left( k\right) \right) }&=0\text{,} \\
\sum_{k=1}^{n}{0\cdot \Delta y\left( k-1\right) }+\frac{1}{n^{2}}%
\sum_{k=1}^{n}{y\left( k\right) f\left( \frac{k}{n},0,0\right) }&=0\text{.}
\end{split}%
\end{equation*}
From $f\left( t,0,0\right) =0$ we have that $x=\theta $ is a solution for $h
= \theta$. Additionally we have uniqueness of solution, so 
\begin{equation*}
x_{\theta }^{n}=\theta .
\end{equation*}%
Hence, from \eqref{normy} and \eqref{ograniczenie} we have 
\begin{equation*}
\begin{split}
\left\Vert x^{n}\right\Vert_{\mathbb{E}} &\leq \frac{1}{n^2}\left\Vert
h\right\Vert_{\mathbb{E}}\leq \frac{2}{n^2}\left\Vert h\right\Vert_{0} = 
\frac{2}{n^2}\left(\sum_{k=1}^{n}{\left|h(k)\right|^2}\right)^\frac{1}{2} \\
&\leq \frac{2}{n^2}\left(\sum_{k=1}^{n}{\left\|h\right\|_{C\left( \left[ 0,1%
\right] \right) }^2}\right)^\frac{1}{2} = \frac{2}{n^2}\left(n\left\|h\right\|_{C\left(\left[ 0,1\right] \right) }^2\right)^\frac{1}{2} = \frac{2}{n^%
\frac{3}{2}}\left\Vert h\right\Vert_{C\left( \left[ 0,1\right] \right) } 
\text{.}
\end{split}%
\end{equation*}
\end{proof}
\begin{theorem}
\label{main_theorem} Assume that conditions \textbf{P1}, \textbf{P2} are
satisfied and that $f:\left[ 0,1\right] \times \mathbb{R}^{2}\rightarrow 
\mathbb{R}$ is continuous. Then there exists $x\in H_{0}^{1}\left(
0,1\right) $ which solves uniquely problem \eqref{problem} and for each $%
n\in \mathbb{N}$ there exists $x^{n}$ which solves uniquely problem %
\eqref{dys}. Moreover relations \eqref{main} are satisfied for the sequence $%
(x^{n})_{n\in \mathbb{N}}$.
\end{theorem}
\begin{proof}
From Theorem \ref{existc} we have existence of a solution to
problem \eqref{problem}, denoted by $x$. Existence of a solution to problem %
\eqref{dys} for each $n\in \mathbb{N}$ is an assertion of Theorem \ref%
{existd}. We denote this solution by $x^{n}$. Let $n\in \mathbb{N}$ be fixed
and let us take any $k\in \left\{ 1, 2,\ldots,n\right\} $. Then we have 
\begin{equation}
\left\vert \Delta x^{n}\left( k-1\right) \right\vert =\left( \left( \Delta
x^{n}\left( k-1\right) \right) ^{2}\right) ^{\frac{1}{2}}\leq \left(
\sum_{k=1}^{n}{\left( \Delta x^{n}\left( k-1\right) \right) ^{2}}\right) ^{%
\frac{1}{2}}=\left\Vert x^{n}\right\Vert _{\mathbb{E}}\text{.}
\label{przeznorme}
\end{equation}%
Multiplying this inequality by $n$ and from \eqref{ogr}, we get 
\begin{equation}
n\left\vert \Delta x^{n}\left( k-1\right) \right\vert \leq n\left\Vert
x^{n}\right\Vert _{\mathbb{E}}\leq n\frac{2}{n^{\frac{3}{2}}}\left\Vert
h\right\Vert _{C\left( \left[ 0,1\right] \right) }=\frac{2}{\sqrt{n}}%
\left\Vert h\right\Vert _{C\left( \left[ 0,1\right] \right) }\leq
2\left\Vert h\right\Vert _{C\left( \left[ 0,1\right] \right) }=:Q
\label{nier2}
\end{equation}%
Additionally. 
\begin{equation}
\left\vert x^{n}\left( k\right) \right\vert =\left\vert \sum_{j=1}^{k}{%
\Delta x^{n}\left( j-1\right) }\right\vert \leq \sum_{j=1}^{k}{\left\vert
\Delta x^{n}\left( j-1\right) \right\vert }  \label{dowod}
\end{equation}%
Moreover from the Cauchy-Schwartz inequality and \eqref{przeznorme} we get%
\begin{equation*}
\begin{split}
\left( \sum_{j=1}^{k}{\left\vert \Delta x^{n}\left( j-1\right) \right\vert }%
\right) ^{2}& =\left( \sum_{j=1}^{k}{\ 1\cdot \left\vert \Delta x^{n}\left(
j-1\right) \right\vert }\right) ^{2}\leq \sum_{i=1}^{k}{\left\vert
1\right\vert ^{2}}\sum_{l=1}^{k}{\left\vert \Delta x^{n}\left( l-1\right)
\right\vert ^{2}} \\
& =k\sum_{i=1}^{k}{\left\vert \Delta x^{n}\left( l-1\right) \right\vert ^{2}}%
\leq k\sum_{l=1}^{n}{\left\vert \Delta x^{n}\left( l-1\right) \right\vert
^{2}}=k\left\Vert x^{n}\right\Vert _{\mathbb{E}}^{2}\text{.}
\end{split}%
\end{equation*}%
Hence 
\begin{equation}
\sum_{j=1}^{k}{\left\vert \Delta x^{n}\left( j-1\right) \right\vert }\leq 
\sqrt{k}\left\Vert x^{n}\right\Vert _{\mathbb{E}}\text{.}  \label{cauchy}
\end{equation}%
Finally from \eqref{dowod} and \eqref{cauchy}, we obtain 
\begin{equation}
\max_{k\in \left\{ 1, 2,\ldots,n\right\} }\left\vert x^{n}\left( k\right)
\right\vert \leq \sqrt{n}\left\Vert x^{n}\right\Vert _{\mathbb{E}}\leq \sqrt{%
n}\frac{2}{n^{\frac{3}{2}}}\left\Vert h\right\Vert _{C\left( \left[ 0,1%
\right] \right) }=\frac{2}{n}\left\Vert h\right\Vert _{C\left( \left[ 0,1%
\right] \right) }\leq 2\left\Vert h\right\Vert _{C\left( \left[ 0,1\right]
\right) }=:N  \label{nier1}
\end{equation}%
Given inequalities \eqref{nier2} and \eqref{nier1} the result now follows
from Lemma 2.4 from \cite{gaines} p. 414, which reads as follows. If we have
solutions to problem \eqref{problem} and problem \eqref{dys} for each $n\in 
\mathbb{N}$ and inequalities 
\begin{equation*}
n|\Delta x^{n}(k-1)|\leq Q\text{, }|x^{n}(k)|\leq N,
\end{equation*}%
then exist subsequence $\left( x^{n}\right) _{n\in \mathbb{N}}$, denoted
still as  $\left( x^{n}\right) _{n\in \mathbb{N}}$, such that 
\begin{equation*}
\lim_{n\rightarrow \infty }\max_{0\leq k\leq n}\left\vert \bar{x}^{n}\left(
k\right) -x\left( \frac{k}{n}\right) \right\vert =0\text{, }%
\lim_{n\rightarrow \infty }\max_{0\leq k\leq n}\left\vert \bar{v}^{n}\left(
k\right) -\dot{x}\left( \frac{k}{n}\right) \right\vert =0\text{,}
\end{equation*}%
where 
\begin{equation*}
\bar{x}^{n}\left( t\right) :=x^{n}\left( k\right) +n\Delta x^{n}\left(
k\right) \left( t-\frac{k}{n}\right) \text{, for }\frac{k}{n}\leq t<\frac{k+1%
}{n}\text{,}
\end{equation*}%
\begin{equation*}
\bar{v}^{n}\left( t\right) :=\left\{ 
\begin{array}{ll}
n\Delta x^{n}\left( k-1\right) +n^{2}\Delta ^{2}x^{n}\left( k-1\right)
\left( t-\frac{k}{n}\right) \text{, } & \frac{k}{n}\leq t<\frac{k+1}{n}\text{%
,} \\ 
n\Delta x^{n}\left( 0\right) \text{, } & 0\leq t<\frac{1}{n}\text{.}%
\end{array}%
\right. 
\end{equation*}%
The mentioned comment from \cite{Tisdell} p. 84 reads that we have this
relation for whole sequence, not only for subsequence.
\end{proof}
\begin{remark}
Note that when $f$ does not depend on the derivative it suffice to obtain
inequality \eqref{nier1}, compare with Lemma 9.3 from \cite{kelly} p. 342.,
which reads as follows. If $f:=f\left( t,x\left( t\right) \right) $ and if
there exists a positive constant $N$, such that 
\begin{equation*}
|x^{n}(k)|\leq N,
\end{equation*}%
then there exists a positive constant $Q$, such that 
\begin{equation*}
n|\Delta x^{n}(k-1)|\leq Q.
\end{equation*}%
Hence we can repeat reasoning from proof of Theorem \ref{main_theorem}.
\end{remark}
\begin{remark}
We observe that in case when $f:=f\left( t,x\left( t\right)\right)$ our main result provides some improvement over its
counterpart from \cite{non-spurious} since we also obtain some information on the convergence of derivatives.
\end{remark}
Concerning the examples of nonlinear terms any function $f$ nondecreasing in 
$x$ and bounded in $\dot{x}$ is of order. See for example
\begin{enumerate}
\item[a)] $f\left( t,\dot{x},x\right) =g\left( t\right) \exp \left(
x-t^{2}\right) \left\vert \arctan \left( \dot{x}\right) \right\vert $,
\item[b)] $f\left( t,\dot{x},x\right) =g\left( t\right) \arctan \left(
x\right) \left\vert \arctan \left( \dot{x}\right) \right\vert $,
\item[c)] $f\left( t,\dot{x},x\right) =g\left( t\right) x^{3}+\exp \left(
x-t^{2}\right) \left\vert \arctan \left( \dot{x}\right) \right\vert $,
\end{enumerate}
where $g:\left[ 0,1\right] \rightarrow \mathbb{R}$ is a continuous and
non-negative function. Recall the Sobolev's inequality $\max_{t\in \left[ 0,1%
\right] }\left\vert x\left( t\right) \right\vert \leq \left(\int_{0}^{1}\dot{%
x}^{2}\left( t\right) dt\right)^\frac{1}{2}$. All above functions satisfy 
\textit{\textbf{P2}} due to standard monotonicity. Assumption \textit{%
\textbf{P1}} is satisfied since in each case we can calculate for a fixed $%
r>0$ a function $f_{r}$. Indeed, by Sobolev's inequality when $\left\Vert
x\right\Vert \leq r$, we see that $\left\vert x\left( t\right) \right\vert
\leq r$ for all $t\in \left[ 0,1\right] $. Then we calculate as follows. Let 
$g_{1}\left( t\right) =g\left( t\right) \exp \left( -t^{2}\right) $ and $%
c_{1}=\max_{\alpha \in \left[ -r,r\right] }\exp \left( \alpha \right) $.
Then, for $\left\Vert x\right\Vert \leq r$, we have 
\begin{equation*}
g\left( t\right) \exp \left( x-t^{2}\right) \left\vert \arctan \left( \dot{x}%
\right) \right\vert \leq \frac{c_{1}\pi }{2}g_{1}\left( t\right)
\end{equation*}%
for all $t\in \left[ 0,1\right] $. The other examples can be demonstrated
likewise.
\subsection{Additional result}
In this section we use the already developed technique to consider the so
called continuous dependence on functional parameters. The idea of the
continuous dependence on parameters is as follows. Let us consider together
with \eqref{problem} a family of boundary value problems 
\begin{equation}  \label{family_n}
\left\{ 
\begin{array}{l}
\ddot{x}\left( t\right) =f\left( t,\dot{x}\left( t\right) ,x\left( t\right)
\right) -h_{n}\left( t\right) , \\ 
x\left( 0\right) =x\left( 1\right) =0,%
\end{array}%
\right.
\end{equation}%
where $f:\left[ 0,1\right] \times \mathbb{R}^{2}\rightarrow \mathbb{R}$ and $%
h_{n}:\left[ 0,1\right] \rightarrow \mathbb{R}$ are continuous functions.
Assume that for any $n$ there exists at least one solution $x_{n}$ to %
\eqref{family_n} and that the sequence $\left( h_{n}\right) _{n\in \mathbb{N}%
}$ is convergent to $h_{0}$ in some sense. The problem \eqref{problem}
depends continuously on a functional parameter $h_{n}$ if sequence $\left(
x_{n}\right) _{n\in \mathbb{N}}$ is convergent, possibly up to a
subsequence, to a solution $x_{0}$ of \eqref{problem} corresponding to $%
h=h_{0}$. To prove continuous dependence on parameter in our problem %
\eqref{problem} we need the following lemma. \newpage
\begin{lemma}
\label{ogr_c} Let $f:\left[ 0,1\right] \times \mathbb{R}^{2}\rightarrow 
\mathbb{R}$ be continuous. By $x$ we denote solution of problem %
\eqref{problem} for functional parameter $h\in H_{0}^{1}\left( 0,1\right) $.
Then we have the inequality 
\begin{equation*}
\left\Vert x\right\Vert \leq \left\Vert h\right\Vert\text{.}
\end{equation*}
\end{lemma}
\begin{proof}
Fix $h\in H_{0}^{1}\left(0,1\right) $.\textbf{\ }Corollary %
\ref{drab}, in which we take $h_{1}=h$, $h_{2}=\theta $ and $c=1$. We have 
\begin{equation*}
\left\Vert x\right\Vert -\left\Vert x_{\theta }\right\Vert \leq \left\Vert
x-x_{\theta }\right\Vert \leq \left\Vert h\right\Vert \text{.}
\end{equation*}%
So 
\begin{equation*}
\left\Vert x\right\Vert \leq \left\Vert h\right\Vert +\left\Vert x_{\theta
}\right\Vert\text{.}
\end{equation*}%
Moreover, for $y\in H_{0}^{1}\left( 0,1\right) $ and $h=\theta $, we get
from \eqref{family_n} 
\begin{equation*}
\int_{0}^{1}{\dot{x}\left(t\right)\dot{y}\left(t\right)dt}%
+\int_{0}^{1}{y\left(t\right)f\left(t, \dot{x}\left(t\right), x\left(t\right)\right)dt}=0.
\end{equation*}
Using $f\left(t,0,0\right) =0$ we see that $x=\theta $ is a solution for $%
h=\theta $. Additionally we have uniqueness of solution, so $x_{\theta
}=\theta .$ Hence, we have 
\begin{equation*}
\left\Vert x\right\Vert \leq \left\Vert h\right\Vert \text{.}
\end{equation*}%
\end{proof}
\begin{theorem}
Assume condition \textit{\textbf{P2}} and that it holds for $t\in \left[ 0,1%
\right] $ 
\begin{equation}  \label{convergence}
f\left(t, \dot{x}\left( t\right), x\left(t\right)\right) =f_{1}\left( t,
x\left(t\right)\right) + \dot{x}\left( t\right) g\left( t\right)\text{,}
\end{equation}%
where $f_{1}:[0,1]\times \mathbb{R}\rightarrow \mathbb{R}$ is continuous and 
$g\in L^{2}\left( 0,1\right) $. Let $\left( h_{n}\right) _{n\in \mathbb{N}%
}\subset H_{0}^{1}\left( 0,1\right) $ be weakly convergent to $h_{0}$. Then
for every $h_{n}$, where $n=0, 1, 2, \ldots$, there exists exactly one solution $%
x_{n}$. Moreover there exists subsequence $\left( x_{n_{k}}\right) _{k\in 
\mathbb{N}}$, which is weakly convergent to $x_{0}$ and $x_{0}$ is solution
for $h_{0}$.
\end{theorem}
\begin{proof}
Let $\left( x_{n}\right) _{n\in \mathbb{N}_{0}}$ be a
sequence of solutions to \eqref{family_n} where each $x_{n}$ corresponds to $h_{n}$. Such
a sequence exists by Theorem \ref{existc}. Note that we do not require
condition \textit{\textbf{P1}} to be satisfied but like in \eqref{weak}, we
have that $\int_{0}^{1}{\dot{x}_{n}\left( t\right) \dot{y}\left( t\right) dt}
$, $\int_{0}^{1}{y\left( t\right) h_{n}\left( t\right) dt}$ are finite for
any $y\in H_{0}^{1}\left( 0,1\right)$. Additionally $y$ is continuous, $x_{n}$, $g\left( t\right) \in
L^{2}\left( 0,1\right) $, so $\int_{0}^{1}{y\left( t\right) \dot{x}%
_{n}\left( t\right) g\left( t\right) }$ is also finite. Finally $x_{n}$ is
continuous and $f_{1}$ is continuous, hence $\int_{0}^{1}{y\left( t\right)
f_{1}\left( t,x_{n}\left( t\right) \right) dt}$ is finite too.
Now we need weak convergence of a subsequence of $\left( x_{n}\right) _{n\in 
\mathbb{N}}$, so we must show that $\left( x_{n}\right) _{n\in \mathbb{N}}$
is bounded. From Lemma \ref{ogr_c} for each $n$, we have 
\begin{equation*}
\left\Vert x_{n}\right\Vert \leq \left\Vert h_{n}\right\Vert \text{.}
\end{equation*}%
Since $\left( h_{n}\right) _{n\in \mathbb{N}}$ is weakly convergent, there
exists positive constant $c$ such that $\left\Vert h_{n}\right\Vert < c$ for all $n\in 
\mathbb{N}$. We finally have that 
\begin{equation*}
\forall _{n\in \mathbb{N}}\left\Vert x_{n}\right\Vert \leq c.
\end{equation*}%
Therefore, we get existence of a weakly convergent subsequence $\left(
x_{n_{k}}\right) _{k\in \mathbb{N}}$. We denote the weak limit of $\left(
x_{n_{k}}\right) _{k\in \mathbb{N}}$ by $x_{0}$. Hence for any  $y\in H_{0}^{1}\left( 0,1\right)$ we have  
\begin{equation*}
\int_{0}^{1}{\dot{x}_{n_{k}}\left( t\right) \dot{y}\left( t\right) dt}%
+\int_{0}^{1}{y\left( t\right) f_{1}\left( t,x_{n_{k}}\left( t\right)
\right) dt}+\int_{0}^{1}{y\left( t\right) \dot{x}_{n_{k}}\left( t\right)
g\left( t\right) }=\int_{0}^{1}{y\left( t\right) h_{n_{k}}\left( t\right) dt}%
\text{.}
\end{equation*}%
Applying the definition of weak convergence and the Lebesgue's Dominated
Convergence Theorem to $\int_{0}^{1}{y\left( t\right) f_{1}\left(
t,x_{n_{k}}\left( t\right) \right) dt}$ and letting $n_{k}\rightarrow
+\infty $ we get 
\begin{equation*}
\int_{0}^{1}{\dot{x}_{0}\left( t\right) \dot{y}\left( t\right) dt}%
+\int_{0}^{1}{y\left( t\right) f_{1}\left( t,x_{0}\left( t\right) \right) dt}%
+\int_{0}^{1}{y\left( t\right) \dot{x}_{0}\left( t\right) g\left( t\right) }%
=\int_{0}^{1}{y\left( t\right) h_{0}\left( t\right) dt}\text{.}
\end{equation*}%
Therefore $x_{0}$ is a solution of \eqref{problem} for $h_{0}$. \end{proof}
We present examples of functions satisfying the above,
\begin{enumerate}
\item[a)] $f\left( t,\dot{x},x\right) =g\left( t\right) \exp \left(
x-t^{2}\right) +g_{1}\left( t\right) \dot{x}$,
\item[b)] $f\left( t,\dot{x},x\right) =g\left( t\right) \arctan \left(
x\right) +g_{1}\left( t\right) \dot{x}$,
\item[c)] $f\left( t,\dot{x},x\right) =g\left( t\right) x^{3}+\exp \left(
x-t^{2}\right) +g_{1}\left( t\right) \dot{x}$,
\end{enumerate}
where $g:\left[ 0,1\right] \rightarrow \mathbb{R}$ is a continuous and
non-negative function and $g_{1}\in L^{2}\left( 0,1\right) $.

\newpage Filip Pietrusiak\newline
Lodz University of Technology\newline
Institute of Mathematics\newline
Wolczanska 215, 90-234 Lodz, Poland\newline
filip.pietrusiak@gmail.com

\end{document}